\documentclass[12pt,oneside]{amsart}
\usepackage[margin=1in]{geometry}

\usepackage{graphicx}
\usepackage{amssymb}
\usepackage{tikz}
\usepackage{amsmath,amsthm,amscd,amssymb}
\usepackage{latexsym}
\usepackage[colorlinks,citecolor=red,pagebackref,hypertexnames=false]{hyperref}
\hypersetup{%
  colorlinks = true,
  citecolor=red,
  linkcolor  = black
}

\usepackage{ulem}
\usepackage{soul}
\numberwithin{equation}{section}
\theoremstyle{plain}
\newtheorem{theorem}{Theorem}[section]
\newtheorem{lemma}[theorem]{Lemma}
\newtheorem{corollary}[theorem]{Corollary}
\newtheorem{proposition}[theorem]{Proposition}

\theoremstyle{definition}
\newtheorem{definition}[theorem]{Definition}

\newtheorem{problem}[theorem]{Problem}
\newtheorem{example}[theorem]{Example}

\theoremstyle{remark}
\newtheorem{remark}[theorem]{Remark}

\newtheorem{case[theorem]}{Case}

\def\R{\mathbb R}

\date{\today}      
\author{A. Iosevich, A. Mayeli, and E. Wyman} 
\address{Department of Mathematics, University of Rochester, Rochester, NY}
\email{iosevich@gmail.com}
\address{Department of Mathematics, City University of New York, New York, NY}
\email{azitamayeli@gmail.com} 
\address{Department of Mathematics, University of Binghamton, Binghamton, NY}
\email{emmett.Wyman@gmail.com}
\thanks{A.I. was supported in part by the National Science Foundation grant DMS - 2154232. A.M. was supported in part by AMS-Simons Research Enhancement Grant and the PSC-CUNY research grants. E.W. was supported in part by the National Science Foundation under grant no. DMS-2422900.}
\thanks{The authors wish to thank B. Kashin for telling them about the reference \cite{GMPT08}. }

\begin{document}

\title[Fourier Uncertainty principles on Riemannian manifolds]{Fourier Uncertainty Principles on Riemannian manifolds}

\begin{abstract} The purpose of this paper is to develop a Fourier uncertainty principle on compact Riemannian manifolds and contrast the underlying ideas with those arising in the setting of locally compact abelian groups. The key obstacle is the growth of eigenfunctions, and connections to Bourgain's celebrated $\Lambda_q$ theorem are discussed in this context. 
\end{abstract}  

\maketitle



\section{Introduction} 
The classical uncertainty principle in Euclidean space says that if a suitably regular function is supported on a set $E \subset {\mathbb R}^d$ and its Fourier transform is concentrated on a set $S \subset {\mathbb R}^d$, then $|E| \cdot |S| \ge c>0$, where $| \cdot |$ denotes the $d$-dimensional Lebesgue measure. See, for example, \cite{DS89} and the references contained therein. The main thrust of this paper is to establish the uncertainty principle and exact recovery mechanisms on Riemannian manifolds, though in the process we shall also establish new results in the well-traveled terrain of locally compact abelian groups. See, for example, \cite{Smith90}, \cite{HJ94} for classical results in that setting. 

\vskip.125in 

\subsection{The uncertainty principle on locally compact abelian groups} While the main thrust of this paper is to study the uncertainty principle on Riemannian manifolds, we set the table by describing this problem in the setting of locally compact abelian groups. Indeed, by Pontryagin duality, if $G$ is a locally compact abelian group, then for the  Haar measure $\mu$ on $G$ (which is unique up to a positive constant), there exists a unique Haar measure $\nu$ on $\widehat{G}$ such that if $f \in L^1(G)$, $\widehat{f} \in L^1(\widehat{G})$, then 
\begin{equation} \label{groupfourierinversion} f(x)=\int_{\widehat{G}} \widehat{f}(\chi) \chi(x) d\nu(\chi), \end{equation} where 
$$ \widehat{f}(\chi)=\int_G f(x) \overline{\chi(x)} d\mu(x),$$ and $\chi\in \widehat G$ is the character\footnote{Recall that a character $\chi$ in $\widehat G$ is a continues group homomorphism function $\chi: G\to S^1$, $S^1$ the unit circle in $\mathbb C$.} on $G$. Note that the equality in \eqref{groupfourierinversion} holds $\mu$-almost everywhere if $f \in L^1(G)$ and point-wise if $f$ is also continuous. The definition of the Fourier transform for functions $f\in L^2(G)$ is done in standard manner, by extending the definition from $L^1(G)\cap L^2(G)$ to $L^2(G)$.  

\smallskip

Recall that we say $f\in L^2(G)$ is supported in a set $E\subset G$ if $f(x)=0$ for $\mu$-almost every $x\in G\setminus E$. In the same fashion, we say $\widehat f$ is supported in $S\subset \widehat G$ if $\widehat f(\xi)=0$ for $\nu$-almost $\xi\in \widehat G\setminus S$. 

\smallskip

Let $E \subset G$ with $0<\mu(E)<\infty$ and $S \subset \widehat{G}$  with $0<\nu(S)<\infty$. 
  Let $f\in L^2(G)$  be supported on $E \subset G$ and $\widehat{f}$ is supported on $S \subset \widehat{G}$, then for $\mu$-almost $x\in E$,  by the Cauchy-Schwarz inequality we have 
\begin{equation} \label{groupUPsetup}  {|f(x)|}^2={\left| \int_S \widehat{f}(\chi) \chi(x) d\nu(\chi) \right|}^2  
  \leq \nu(S) \cdot \int_S {|\widehat{f}(\chi)|}^2 d\nu(\chi). 
\end{equation}

Integrating both sides of (\ref{groupUPsetup}) over $E$ (where $f$ is assumed to be supported on $E$), we see that 
$$ {||f||}^2_{L^2(E)} \leq \nu(S) \cdot \mu(E) \cdot \int_S {|\widehat{f}(\chi)|}^2 d\nu(\chi) = \nu(S) \cdot \mu(E) \cdot {\|f\|}^2_{L^2(E)}$$ by the Plancherel identity.
From this  inequality  we conclude that   the following variant  of uncertainty principle: 
\begin{equation} \label{UPforLCA} \mu(E) \cdot \nu(S) \ge 1. \end{equation} 

The assumption that $f$ is supported in $E$ and $\widehat{f}$ is supported in $S$ can easily be relaxed  and replaced by the assumption that $f, \widehat{f}$ are concentrated in those sets in a suitable sense. The following definition will be used throughout the paper in different settings.

\begin{definition}\label{def: L^p concentrated} If $X$ is a measure space with measure $\mu$ and $E \subset X$, we say that $f\in L^p(X,\mu)$ is $L^p$-concentrated on $E$ at level $L \ge 1$ if 
\begin{equation} \label{Level-concenteration} {\left( \int_X {|f(x)|}^p d\mu(x) \right)}^{\frac{1}{p}} \leq L {\left( \int_E {|f(x)|}^p d\mu(x) \right)}^{\frac{1}{p}}. \end{equation} 
\end{definition} 

It is not difficult to see that (\ref{Level-concenteration}) is equivalent to the condition 
\begin{equation} \label{Level-concenteration-equivalent} {\|f-1_Ef\|}_{L^p(X)} \leq \epsilon {\|f\|}_{L^p(X)}\end{equation} for some $0 \leq \epsilon<1$. Indeed (\ref{Level-concenteration}) holds with $L=\frac{1}{{(1-\epsilon^p)}^{\frac{1}{p}}}$ if and only if (\ref{Level-concenteration-equivalent}) holds. 

Here, $1_E$ is the indicator of the set $E$ and is defined by $1_E(x)=1$ when $x\in E$ and $0$ otherwise.  
Note that any measurable function in $L^p$ that is supported on a measurable set is automatically $L^p$-concentrated on the set at level $L=1$. This explains why concentration can be seen as a relaxed version of support. 

\vskip.12in 
If $\widehat{f}$ is supported in a generic set (see Remark \ref{remarkgeneric}), the situation improves significantly. To see this, we recall the following result due to  Bourgain (\cite{Bourgain89}). 

\begin{theorem} \label{bourgaintheorem} Let $G$ be a locally compact abelian group and let $\Psi=(\psi_1, \dots, \psi_n)$ denote a sequence of $n$ mutually orthogonal functions, with ${\|\psi_i\|}_{L^{\infty}(G)} \leq 1$. 
Then for any $q>2$, 
there exists an index  subset $S$ of $\{1,2, \dots, n\}$, with size  $|S|>n^{\frac{2}{q}}$,  such that 
$$ {\left\| \sum_{i \in S} a_i \psi_i \right\|}_{L^q(G)} \leq C(q) {\left( \sum_{i \in S} {|a_i|}^2 \right)}^{\frac{1}{2}}.$$ 

The constant $C(q)$ depends only on $q$ and the estimate above holds for a generic index set of size $\lceil n^{\frac{2}{q}} \rceil$, where $\lceil x\rceil$ denotes the smallest integer greater than $x$. 
\end{theorem} 

\begin{remark} Note that Theorem \ref{bourgaintheorem} holds automatically for $q=2$.  In this case, $C(2)= 1$ (by orthogonality) for any set $S \subset \{1,2, \dots, n\}$. When $q=\infty$, it is easy to see that  $C(q)\leq |S|^{\frac{1}{2}}$ by Cauchy-Schwarz and the assumption ${||\psi_i||}_{L^{\infty}(G)} \leq 1$. By Riesz-Thorin interpolation theorem applied to the identity operator on the Paley-Wiener subspace of $L^2(G)$ generated by 
${\{\psi_i\}}_{i \in S}$, $C(q) \leq {|S|}^{\frac{1}{2}-\frac{1}{q}}$. The point of Theorem \ref{bourgaintheorem} is that $C(q)$ depends only on $q$ if $S$ is a random subset of $\{1,2, \dots, n\}$ of size $\lceil n^{\frac{2}{q}} \rceil$. \end{remark}

\begin{remark}\label{remarkgeneric}   
The notion of generic in Theorem \ref{bourgaintheorem} means the following. Let $0<\delta<1$ and let ${\{\xi_j\}}_{1 \leq j \leq n}$ denote independent $0,1$ random variables of mean $\int \xi_j(\omega) d\omega=\delta$, $1 \leq j \leq n$. Choosing $\delta=n^{\frac{2}{q}-1}$ generates a random subset $S_{\omega}=\{1 \leq j \leq n: \xi_j(\omega)=1\}$ of $\{1,2, \dots n\}$ of expected size $\lceil n^{\frac{2}{q}} \rceil$. Theorem \ref{bourgaintheorem} holding for a generic set $S$ means that the result holds for the set $S_{\omega}$ with probability $1-o_N(1)$. In simpler language, if we randomly choose a subset of $\{1,2, \dots, n\}$ by choosing each element with probability $p=n^{\frac{2}{q}-1}$, then Theorem \ref{bourgaintheorem} holds for such a set with probability close to $1$. 
\end{remark} 

Theorem \ref{bourgaintheorem} has a built-in uncertainty principle mechanism. Indeed, we have the following theorem, which is the first result of this paper. 

\begin{theorem} \label{bourgainup} Let $G$ be a locally compact abelian group with a Haar measure $\mu$, and let $\Psi=(\psi_1, \dots, \psi_n)$ denote a sequence of $n$ mutually orthogonal functions, with ${||\psi_i||}_{L^{\infty}(G)} \leq 1$. Let $q>2$, and let $S$ be a generic index  subset of $\{1,2, \dots, n\}$ of size $\lceil n^{\frac{2}{q}} \rceil$. Define   $f(x)=\sum_{i \in S} a_i \psi_i(x)$, for almost every  $x\in G$. Suppose that $f$ is $L^2$-concentrated on $E \subset G$ at level $L$. Then 
$$ \mu(E) \ge \frac{1}{{(LC(q))}^{\frac{1}{\frac{1}{2}-\frac{1}{q}}}}.$$ 
\end{theorem}







\vskip.125in 

In particular, the results stated above hold for ${\mathbb T}^d$, the torus. This case brings us into the realm of Riemannian manifolds. The key assumption in the arguments above is that the orthogonal functions have modulus at most one, which underlines the fundamental difficulty of the manifold setting where the results of this paper will be stated. Recall that if $M$ is a compact Riemannian manifold without a boundary, then the eigenfunctions $\{e_j\}_{j=1}^{\infty}$ of $\sqrt{-\Delta_g}$ form an orthonormal basis for $L^2(M)$. As we shall see, the reason the argument given for locally compact abelian groups above does not transfer directly to the manifold setting is that the eigenfunctions of the Laplacian are not typically uniformly bounded. 

\vskip.125in 




\subsection{The review of basic facts regarding Riemannian manifolds}


Let $M$ be a smooth $d$-dimensional manifold which is compact and without boundary, and let $g$ be a Riemannian metric tensor on $M$. The metric tensor $g$ is locally represented as a real-symmetric, positive-definite $d \times d$ matrix with entries that are smooth functions of $x$,
\[
    g(x) = \begin{bmatrix}
        g_{11}(x) & g_{12}(x) & \cdots & g_{1d}(x) \\
        g_{21}(x) & g_{22}(x) & \cdots & g_{2d}(x) \\
        \vdots & \vdots & & \vdots \\
        g_{d1}(x) & g_{d2}(x) & \cdots & g_{dd}(x)
    \end{bmatrix}
\]
where, if $\partial_i = \frac{\partial}{\partial x_i}$, we have
\[
    g_{ij}(x) = \langle \partial_i, \partial_j \rangle_{g(x)}.
\]

\vskip.125in 

The natural volume density on a Riemannian manifold $(M,g)$ is given locally by 
$$dV_g(x) = |g(x)|^{\frac{1}{2}} dx,$$ where $|g(x)|$ denotes the determinant of $g(x)$.

\vskip.125in 

There is a natural generalization of the Euclidean Laplacian to $(M,g)$, called the Laplace-Beltrami operator, which is expressed in local coordinates by
\[
    \Delta_g f = |g|^{-\frac{1}{2}} \sum_{i,j = 1}^n \partial_i (|g|^{\frac{1}{2}} g^{ij} \partial_j f).
\]
Here $g^{ij}$ are the matrix entries of $g^{-1}$. Indeed, the Laplace-Beltrami operator is defined such that
\begin{align*}\label{def:LBO}
    \int_M (\Delta_g f) h \, dV_g = -\int_M \langle \nabla_g f, \nabla_g h \rangle_g \, dV_g,
\end{align*}
where
\[
    \nabla_g f(x) = \sum_{i,j = 1}^d g^{ij}(x) \partial_i f(x) \partial_j
\]
denotes the gradient of $f$. It follows $\Delta_g$ is both self-adjoint and negative-definite.

\subsection{Laplace-Beltrami eigenfunctions and eigenvalues}

We now summarize some facts about the spectrum of $\Delta_g$. The material is standard and can be found, for example, in \cite{J02}. There exists an orthonormal basis for $L^2(dV_g)$ of eigenfunctions $e_1, e_2, \ldots$ satisfying
\[
    \Delta_g e_j = -\lambda_j^2 e_j, \qquad j = 1,2,\ldots
\]
where $\lambda_j \geq 0$. We are free to select the indices of our eigenfunctions to ensure
\[
    0 = \lambda_1 \leq  \lambda_2 \leq \lambda_3 \leq \cdots \to \infty.
\]
Most of these facts follow by applying the spectral theory of compact self-adjoint operators to the resolvent $(I - \Delta_g)^{-1}$.

\begin{remark} The choice to use $-\lambda_j^2$ as the eigenvalue for $e_j$ is a matter of convention. This allows us to think of $\lambda_j$ as the ``frequency" of $e_j$, or rather the frequency of the standing wave solution
\[
    u(t,x) = \cos(t\lambda_j) e_j(x)
\]
to the wave equation $(\Delta_g - \partial_t^2) u = 0$.
\end{remark}

Given a function $f\in L^2(dV_g)$, the  Fourier coefficients of $f$  are given by $\widehat f(j) = \langle f, e_j \rangle$, the Hermitian inner product of $f$ with $e_j$ in $L^2(dV_g)$, given by 
\[
    \langle f, e_j \rangle = \int_M f(x) \overline{e_j(x)} \, dV_g(x).
\]
Indeed, if $f$ is in $L^2(dV_g)$, then 
\[
    f = \sum_{j = 1}^\infty \widehat f(j) e_j
\]
where both equality and the convergence of the sum hold in $L^2(dV_g)$. If $f$ is smooth enough, the sum converges uniformly and equality holds pointwise.  

\vskip.125in

A fundamental problem in the harmonic analysis on manifolds is to estimate the count of eigenvalues (with multiplicity) under some threshold. That is, we desire nice asymptotics for the Weyl counting function
\[
    N(\lambda) = \#\{j \in \mathbb N : \lambda_j \leq \lambda \}.
\]
The Weyl law gives very general asymptotics
\begin{equation}\label{eq: Weyl law}
    N(\lambda) = (2\pi)^{-d} |M| |B| \lambda^d + O(\lambda^{d-1}),
\end{equation}
where $|M|$ denotes the volume of $M$ and $|B|$ denotes the volume of the unit ball in $\mathbb R^d$. The remainder is sharp in the sense that it cannot be improved for the sphere $S^d$. Many manifolds enjoy an improved Weyl remainder estimate (\cite{DG75}). 

The flat torus $M = \mathbb R^d / {(2\pi \mathbb Z)}^d$ provides a key illustrative example. The Laplace-Beltrami operator  is just $\sum_{\ell=1}^d  \partial
_ \ell^2$. 
We have a complete, $\mathbb Z^d$-indexed set of eigenfunctions
\[
    e_m(x) = (2\pi)^{-\frac{d}{2}} e^{i\langle x, m \rangle}, \qquad m \in \mathbb Z^d
\]
with $m\in \Bbb Z^d$ and $\lambda_m = |m|^2$, respectively. Here, $|m|^2= m_1^2+\cdots+m_d^2$. 

\vskip.125in 

The classical Weyl Law (see e.g. \cite{soggeFIOs}) says that as $\lambda\to \infty$
\[
    N(\lambda) = \#\{m \in \mathbb Z^d : |m| \leq \lambda \} = |B| \lambda^d + O(\lambda^{d-1}),
\]
which coincides with the trivial bound for the Gauss circle problem. 
We will need the {\it local Weyl counting function},


\[
    N_x(\lambda) = \sum_{\lambda_j \leq \lambda} |e_j(x)|^2, \qquad x \in M, \ \lambda \geq 0,
\]
which enjoys similar asymptotics to the Weyl counting function, namely
\[
    N_x(\lambda) = (2\pi)^{-d} |B| \lambda^d + O(\lambda^{d-1}),  \ \ \text{as} \ \lambda\to \infty
\]
where the constants implicit in the big-$O$ remainder are uniform for $x \in M$. We note the pointwise asymptotics imply the Weyl law since $\|e_j\|_{L^2} = 1$ and
\[
    N(\lambda) = \int_M N_x(\lambda) \, dV_g(x).
\]

\subsection{An uncertainty principle for compact manifolds without boundary}

We begin a general technical statement that we shall proceed to clarify in a variety of instances. 


\begin{proposition}\label{prop: l2 manifold uncertainty} 
Let $S$ be a finite subset of the set of eigenvalues of $\sqrt{-\Delta_g}$. Let $X_S=\{j: \lambda_j \in S\}$. Suppose that $f \in L^2(M)$ is not identically $0$ and $f$ is $L^2$-concentrated in $E \subset M$ at level $L$ with respect to the Riemannian volume density (see Definition \ref{def: L^p concentrated}). Suppose also that $\widehat{f}$ is $L^2$-concentrated on $X_S$ at level $L'$ with respect to the counting measure. Then 
\begin{equation} \label{funnyquantity}
        \left( \frac{1}{\# X_S} \sum_{j \in X_S} \frac{1}{|E|} \int_E |e_j(x)|^2 \, dx \right)^{-1} \leq (1 - \epsilon - \epsilon')^{-2} |E| (\# X_S)
    \end{equation}
    where
\begin{equation}\label{eq: L-epsilon}
        L = (1 - \epsilon^2)^{-\frac 1 2} \qquad \text{ and } \qquad L' = (1 - \epsilon'^2)^{-\frac 1 2}.
    \end{equation}
\end{proposition}

\begin{remark}\label{remark: operators}
    This result is valid for general orthogonal systems in $L^2(M)$  and holds without the assumption that the underlying manifold does not have a boundary. See, for example, \cite{Bell1}, \cite{ArieAzita23}, and \cite{Bell2} for some interesting implementations of this viewpoint.  
\end{remark}

\vskip.125in 



Proposition \ref{prop: l2 manifold uncertainty} raises the question of when the left-hand side of (\ref{funnyquantity}) is bounded below by a non-zero constant.  Suppose that for each frequency $\lambda$, the sum
\begin{align}\label{constant-level-function}
    \sum_{j : \lambda_j = \lambda} |e_j(x)|^2 \equiv c \quad \text{a.e.} \ 
    x\in M 
\end{align}
is constant,  and the constant $c:=c(\lambda, M)$   is independent of  $x$. Since the $e_j$'s are all $L^2$-normalized, this condition is equivalent to
\begin{equation} \label{eq: constant l2 condition}
    \sum_{j : \lambda_j = \lambda} |e_j(x)|^2 \equiv \frac{\#\{j : \lambda_j = \lambda \}}{|M|}, \quad \text{a.e.} \ 
    x\in M
\end{equation}
where we count  in \eqref{eq: constant l2 condition}  with   multiplicity.  
Indeed, notice that by  \eqref{constant-level-function}, we have $$\sum_{j : \lambda_j = \lambda} \int_M |e_j(x)| dx = c |M| .$$    On the other hand,   by the $L^2$-normalization of $e_j$, we have 
$$   \sum_{j : \lambda_j = \lambda} \int_M |e_j(x)| dx = \#\{j : \lambda_j = \lambda \}.$$
Putting the last two equations  together, we obtain $c = \frac{\#\{j : \lambda_j = \lambda \}}{|M|}$. Recalling the value of the constant      $c$ from \eqref{constant-level-function},   we complete the proof.

\bigskip

All homogeneous Riemannian manifolds - those whose group of isometries have a transitive action on the points in the manifold - satisfy this condition. In dimension $d=2$, the property \eqref{eq: constant l2 condition} holds if and only if the manifold $M$ is homogeneous \cite{homogeneous}. This question is still open in higher dimensions. We note that flat tori and spheres with the standard metrics are all homogeneous, so  the property  \eqref{eq: constant l2 condition} holds in these cases.  

   Notice that the eigenfunctions on the sphere are not uniformly bounded; however, the relation \eqref{eq: constant l2 condition} holds due to the homogeneity of the sphere. Furthermore, the sphere is not a locally compact abelian group, which makes it a particularly important and interesting object in this context. 
   



\begin{corollary}\label{cor: homogeneous l2 manifold uncertainty}
    Let $(M,g)$ be a compact  Riemannian manifold without boundary satisfying the relation \eqref{eq: constant l2 condition}. Let $E$ be a measurable subset of $M$ and let $S$ be a finite subset of the spectrum of $\sqrt{-\Delta_g}$. Let $X_S=\{j: \lambda_j \in S\}$. Suppose $f$ is $L^2$-concentrated in $E$ at level $L$ with respect to the Riemannian volume measure, and $\widehat f$ is $L^2$-concentrated in $X_S$ at level $L'$ with respect to the counting measure. 
    Then,
    \[
        (1 - \epsilon - \epsilon')^2 \leq \frac{|E|}{|M|} \cdot \# X_S,
    \]
    where $\epsilon$ and $\epsilon'$ are as in \eqref{eq: L-epsilon}.
\end{corollary}

\vskip.125in 

In the case when the manifold does not satisfy (\ref{eq: constant l2 condition}), we are still able to prove an uncertainty principle that depends on the size of the frequency concentration set $S$. The following corollary is an immediate consequence  of  Proposition \ref{prop: l2 manifold uncertainty}. 

\begin{corollary} \label{cor: size-dependent l2 manifold uncertainty} 
    Assume the hypotheses of Proposition \ref{prop: l2 manifold uncertainty}. Then,
    \begin{equation} \label{equation: classicaluncertaintyonmanifolds}
        (1 - \epsilon - \epsilon')^2
        \leq |E| \left( \sup_{x \in M} \sum_{\lambda_j \in S} |e_j(x)|^2 \right).
    \end{equation} 
\end{corollary} 
Note that  the supremum on $M$ can actually be replaced by the supremum on $E$. 

\vskip.1in 

To be useful,  this corollary requires as input some estimates on the sup-norms of eigenfunctions. For large $\lambda>0$, the general bounds are
\begin{equation} \label{equation: soggezelditch}
    \sup_{x \in M} \sum_{\lambda_j \in [\lambda, \lambda + 1]} |e_j(x)|^2 = O(\lambda^{d-1}) \quad \text{as}  \quad \lambda\to \infty.
\end{equation} 
(see \cite{soggeFIOs} and the references therein).
These bounds are sharp, but can be improved if one both makes some assumptions on the geodesic flow on $M$ and replaces the unit interval $[\lambda, \lambda + 1]$ with an interval $[\lambda, \lambda + o(1)]$ that slowly shrinks in length as $\lambda \to \infty$ (see e.g. \cite{berard,SZDuke}).

\vskip.125in 


Corollary \ref{cor: size-dependent l2 manifold uncertainty} and the bound (\ref{equation: soggezelditch}) lead us to the following result. 

\begin{corollary} \label{theorem: coveringuncertainty}

Assume the hypotheses of Proposition \ref{prop: l2 manifold uncertainty} and suppose that $S$ can be covered by unit intervals $\cup_{k=1}^n [\mu_k, \mu_k+1]$, i.e., 
$S\subset \cup_{k=1}^n [\mu_k, \mu_k+1]$. 
Then 
\begin{equation} \label{equation: generaluncertaintyprinciple} 
    (1 - \epsilon - \epsilon')^2 \leq |E| \cdot C_M \sum_{k=1}^n \mu_k^{d-1}, 
\end{equation} 
where $C_M$ is the constant implicit on the right-hand side of the equation \eqref{equation: soggezelditch}.

\end{corollary} 

\vskip.125in 
%

\begin{remark}   
The quantity on the right-hand side of (\ref{equation: generaluncertaintyprinciple}) carries a lot of information. 
Assume that $\lambda_m= \max\{\lambda\in S\}$.
We always have 
\begin{equation} \label{equation: stupidbound} 
    \sum_{k=1}^n \mu_k^{d-1} \leq \# S \cdot   (\lambda_m + 1)^{d-1}, 
\end{equation}
since $\# S\geq n$. 
On the other hand, suppose that $\mu_k \approx 2^k$. Then $\lambda_m\leq   2^{N}$, and we obtain
$$ \sum_{k=1}^n \mu_k^{d-1} \approx \sum_{k=1}^N 2^{k(d-1)} \leq 2^{N(d-1) + 1},$$ which is much better than \eqref{equation: stupidbound}.
\end{remark}

\subsection{Functions on Riemannian manifolds with random spectra}

The following is a manifold analog of Theorem \ref{bourgainup}. 

\begin{theorem} \label{theorem: randomshitonmanifolds} Let $M$ be a compact Riemannian manifold without a boundary, and let $f \in L^2(M)$. Let $\{e_1, e_2, \dots, e_n\}$ denote a subset of the eigenbasis of $\Delta_g$ such that 
$$ {\|e_i\|}_{\infty} \leq B.$$ 
Then there exists $C,c>0$ such that for a generic index subset $I \subset \{1, 2, \dots, n\}$ with 
\[ \left| \# I-   \frac{n}{2}\right|  
 \leq c \sqrt{n}
\]
 such that if 
$$ f(x)=\sum_{i \in I} a_i e_i \quad  \text{is $L^1$-concentrated in} \ E \subset M \ \text{at level} \ A, $$ then 
$$ |E| \ge \frac{1}{C^2 \cdot B^2 \cdot A^2 \cdot \log(n)^2 \cdot {(\log(\log(n))}^5} .$$
\end{theorem}

\vskip.125in 

We now compare Theorem \ref{theorem: randomshitonmanifolds} to the results obtained above, specifically in the case where $S$ is a generic subset of the first $n$ eigenvalues  
\[
    \lambda_1 \leq \lambda_2 \leq \cdots \leq \lambda_n
\]
repeated with multiplicity. The Weyl law \eqref{eq: Weyl law} guarantees positive constants $c$ and $C$ for which 
\[
    c\lambda_n^d \leq n \leq C \lambda_n^d \qquad \text{ for large $n$},
\]
and hence $\lambda_n \approx n^\frac 1 d$. The corresponding eigenfunctions $e_1, \ldots, e_n$ have $L^\infty$ norms uniformly bounded by $O(\lambda_n^\frac{d-1}{2})$, which can often be improved (see e.g. \cite{berard, SZDuke, geodesicbeams}). The lower bound in the theorem then reads as
\begin{equation} \label{eq: randombound}
    |E| \geq \frac 1 {C^2 \cdot A^2 \cdot \lambda_n^{d-1} \cdot \log(\lambda_n)^2\cdot (\log (\log (\lambda_n)))^5},
\end{equation} 
after perhaps adjusting the constant $C$. On the other hand, Corollary \ref{theorem: coveringuncertainty} yields the bound 
\[
    |E| \geq \frac 1 {C^2 \cdot A^2 \cdot \lambda_n^d}
\]
in this situation, which is worse than (\ref{eq: randombound}) by nearly a full power of $\lambda_n$. 

\vskip.25in 

\subsection{Jointly elliptic operators} We now bridge the conceptual gap between the Fourier uncertainty principal on manifolds and on locally compact abelian groups. The torus, being a model member of both settings, will be our focal point.

Let $M$ be a compact smooth manifold without boundary which comes equipped with a volume density which we will simply denote as $dx$. Suppose we have some differential operators $L_1, \ldots, L_m$, which together have the following properties.
\begin{itemize}
    \item[(i)] $L_i$ is self-adjoint in $L^2(dx)$ for each $i$.
    \item[(ii)] $L_i$ and $L_j$ commute for each $i,j$.
    \item[(iii)] $L_1, \ldots, L_m$ are jointly elliptic in the sense that
    \[
        T := I + \sum_{i = 1}^m L_i^2
    \]
    is an elliptic differential operator on $M$.
\end{itemize}

Then, the tuple of operators $L = (L_1, \ldots, L_m)$ admits a discrete orthonormal basis $\{e_j : j \in \mathbb N\}$ of joint eigenfunctions satisfying
\[
    L_i e_j = \lambda_{i,j} e_j \qquad \text{ for each } i \in \{1, \ldots, m\}, \ j \in \mathbb N.
\]
The joint spectrum of this family of operators is
\[
    \{ (\lambda_{1,j}, \ldots, \lambda_{m,j}) \in \mathbb R^m : \ j \in \mathbb N\}.
\]
We we will feel free to replace $\mathbb N$ with any other countably infinite indexing set.

\begin{example}[The torus]
    Consider the flat torus $\mathbb T^d = \R^d / 2\pi \mathbb Z^d$ with operators
    \[
        L_i = \frac 1 i \frac{\partial}{\partial x_i} \qquad \text{ for } i \in \{1, \ldots, d\}.
    \]
    Note $\{L_1, \ldots, L_d\}$ is a jointly elliptic, commuting family of self-adjoint operators as described above, which also admits a joint eigenbasis of Fourier exponentials
    \[
        e_n(x) = (2\pi)^{-\frac d 2} e^{i\langle x, n \rangle} \qquad  n \in \mathbb Z^d.
    \]
    As expected, the joint spectrum is $\mathbb Z^d$.
\end{example}

\begin{example}[Cartesian products of manifolds]
    Let $M_1$ and $M_2$ denote compact Riemannian manifolds without boundary with respective Laplace-Beltrami operators $\Delta_1$ and $\Delta_2$. By an abuse of notation, we consider $\Delta_1$ and $\Delta_2$ as operators acting on functions on $M_1 \times M_2$ in their respective variables. Note, $\{\Delta_1, \Delta_2\}$ is a family of commuting, jointly elliptic, self-adjoint operators. The joint eigenbasis consists of elements, each a tensor product of an eigenfunction on $M_1$ with one on $M_2$. The joint spectrum is $\Lambda_1 \times \Lambda_2$, where $\Lambda_1$ and $\Lambda_2$ are the spectra of $\Delta_1$ and $\Delta_2$, respectively.
\end{example}

\begin{example}[The two-dimensional sphere]
    The classic spherical harmonics $Y_\ell^m$ on $S^2$ are joint eigenbasis elements for family 
    \[
        \left\{ \frac 1 i \frac{\partial}{\partial \theta}, -\Delta_{S^2} \right\},
    \]
    where $\theta$ here denotes the azimuthal angle. The joint spectrum is
    \[
        \{ (m, \ell(\ell+1) ) : |m| \leq \ell \}.
    \]
\end{example}

The statement for Proposition \ref{prop: l2 manifold uncertainty} can be phrased and proved in nearly exactly the same way, with $S$ being a subset of the joint spectrum of a family of operators. 

\begin{proposition}\label{prop: joint}
    Let $M$ be a compact Riemannian manifold with a system of self-adjoint, commuting, jointly-elliptic differential operators $L_1, \ldots, L_m$ with joint eigenbasis $e_j$ for $j \in \mathbb N$ and joint spectrum $\Lambda \subset \R^m$. Let $E \subset M$ be measurable, let $S$ be a finite subset of the joint spectrum $\Lambda$. Let $X_S=\{j: \lambda_j \in S\}$, where $\lambda_j$ is the vector of eigenvalues of $L_1, \dots, L_m$ associated with the eigenfunction $e_j$. Suppose there is a nontrivial $f \in L^2(M)$ that is $L^2$-concentrated in $E$ at level $L$ and whose Fourier coefficients $\widehat f$ are $L^2$-concentrated in $X_S$ at level $L'$. Then 
    \[
        \sum_{j \in X_S} \int_E |e_j(x)|^2 \, dx \geq (1 - \epsilon - \epsilon')^2,
    \]
    where
    \[
         L = (1 - \epsilon^2)^{-\frac 1 2} \qquad \text{ and } \qquad L' = (1 - \epsilon'^2)^{-\frac 1 2}.
    \]
\end{proposition}

\vskip.125in 

The left-hand side can be simplified as in Corollary \ref{cor: homogeneous l2 manifold uncertainty} provided
\begin{equation}\label{eq: joint homogeneous}
    \sum_{j : \lambda_j = \lambda} |e_j(x)|^2 \ \text{is constant in } x 
\end{equation} for each $\lambda$ in the joint spectrum. As we saw earlier in  \eqref{eq: constant l2 condition},   in this case, we have   
$$  
    \sum_{j : \lambda_j = \lambda} |e_j(x)|^2 = \frac{\#\{j : \lambda_j = \lambda\}}{|M|}.
$$  

\begin{corollary}
    Assume the hypotheses of Proposition \ref{prop: joint}, and further assume \eqref{eq: joint homogeneous}. Then,
    \[
        \# X_S \frac{|E|}{|M|} \geq (1 - \epsilon - \epsilon')^2, 
    \]
    where $X_S$ is as described in Proposition \ref{prop: joint}. 
\end{corollary}

\section{Proofs of main results}

\vskip.125in 

In this section we prove the main results of the paper. 

\subsection{Proof of Theorem \ref{bourgainup}
}

In this subsection, we shall use the notation $|E|=\mu(E)$,  where $\mu$ is the Haar measure on $G$.
By Theorem \ref{bourgaintheorem}, 
\begin{equation} \label{bourgainprep} {\|f\|}_{L^q(G)} \leq C(q) {\|f\|}_{L^2(G)}.  \end{equation}

\vskip.125in 

The left-hand side of (\ref{bourgainprep}) is bounded from below by  

\begin{equation} \label{lhsprep} {\left( \int_E {|f(x)|}^q d\mu(x) \right)}^{\frac{1}{q}}={|E|}^{\frac{1}{q}} \cdot {\left( \frac{1}{|E|} \int_E {|f(x)|}^q d\mu(x) \right)}^{\frac{1}{q}}. \end{equation}  

Since $f$ is assumed to be $L^2$ concentrated on $E$ at level $L$,  the right-hand side of (\ref{bourgainprep}) is bounded from above by    
\begin{equation} \label{rhsprep} C(q) \cdot L \cdot {\left( \int_E {|f(x)|}^2 d\mu(x) \right)}^{\frac{1}{2}}=L \cdot C(q) \cdot {|E|}^{\frac{1}{2}} \cdot {\left( \frac{1}{|E|} \int_E {|f(x)|}^2 d\mu(x) \right)}^{\frac{1}{2}}. \end{equation} 
Plugging (\ref{lhsprep}) and (\ref{rhsprep}) into (\ref{bourgainprep}),    we obtain 
$$ {|E|}^{\frac{1}{q}} \cdot {\left( \frac{1}{|E|} \int_E {|f(x)|}^q d\mu(x) \right)}^{\frac{1}{q}} \leq C(q) L \cdot {|E|}^{\frac{1}{2}} \cdot {\left( \frac{1}{|E|} \int_E {|f(x)|}^2 d\mu(x) \right)}^{\frac{1}{2}}.$$
It follows that 
$$ {|E|}^{\frac{1}{2}-\frac{1}{q}} \ge \frac{1}{LC(q)} \cdot \frac{{\left( \frac{1}{|E|} \int_E {|f(x)|}^q d\mu(x) \right)}^{\frac{1}{q}}}{{\left( \frac{1}{|E|} \int_E {|f(x)|}^2 d\mu(x)\right)}^{\frac{1}{2}}} \ge \frac{1}{LC(q)}.$$
We conclude that 
$$ |E| \ge \frac{1}{{(LC_q)}^{\frac{1}{\frac{1}{2}-\frac{1}{q}}}} ,$$ as desired. 

\vskip.125in 

\subsection{Proof of Proposition \ref{prop: l2 manifold uncertainty}} 
To prove the proposition, we first need to introduce some notation: 

We denote by $P_E$ the cut-off map $f\to 1_E f$,  and by $B_S$ the band-limiting projection  
$$
    B_S(f) = \sum_{\lambda_j \in S} \widehat f(j) e_j = \sum_{\lambda_j \in S} \left( \int_M f(x) \overline{e_j(x)} \, d\mu(x) \right) e_j .
$$
Proposition \ref{prop: l2 manifold uncertainty} follows immediately from the next two lemmas, and the assumption that $P_EB_Sf\neq 0$. As per Remark \ref{remark: operators}, the statements and proofs only require that $\{e_j : j \in \mathbb N\}$ be a set of   orthonormal  functions 
 in   $L^2(M)$.

\begin{lemma}\label{lem: lemma 1}
    Suppose $f$ is a $L^2$-concentrated on $E$ at level $L$ and that $\widehat f$ is $\ell^2$-concentrated on $S$ at level $L'$. Then,
    \[
         (1 - \epsilon - \epsilon') \|f\|_{L^2} \leq \|P_E B_S f\|_{L^2}.
    \]
    where $L = (1 - \epsilon^2)^{-\frac 1 2}$ and $L' = (1 - \epsilon'^2)^{-\frac 1 2}$ as per Definition \ref{def: L^p concentrated}.
\end{lemma}

\begin{lemma}\label{lem: lemma 2} With the hypothesis of the previous lemma, we have 
    \[
        \|P_E B_S f\|_{L^2} \leq \left( \int_E \sum_{\lambda_j \in S} |e_j(x)|^2 \, dx \right)^\frac 1 2 \|f\|_{L^2}
    \]

\end{lemma}

\begin{proof}[Proof of \ref{lem: lemma 1}]
    Note
    \[
        \|P_E f\|_{L^2} \leq \|f\|_{L^2} \qquad \text{ and } \qquad \|B_S f\|_{L^2} \leq \|f\|_{L^2} \qquad \text{ for all } f \in L^2.
    \]
    Hence,
    \begin{align*}
        0 &\leq \|f\|_{L^2} - \|P_E B_S f\|_{L^2} \\
        &\leq \|f - P_E B_S f\|_{L^2} \\
        &\leq \|f - P_E f\|_{L^2} + \|P_E f - P_E B_S f\|_{L^2} \\
        &\leq \|f - P_E f\|_{L^2} + \|f - B_S f\|_{L^2}.
    \end{align*}
    By hypothesis, $f$ satisfies
    \[
        \|f - P_E f\|_{L^2} \leq \epsilon \|f\|_{L^2} \qquad \text{ and } \qquad \|f - B_S f \|_{L^2} \leq \epsilon' \|f\|_{L^2},
    \]
    from which the lemma follows.
\end{proof}


\begin{proof}[Proof of Lemma \ref{lem: lemma 2}]
    We have
    \begin{align*}
        \|P_E B_S f\|_{L^2}^2 &= \int \left| 1_E(x) \sum_{\lambda_j \in S} \widehat f(j) e_j(x) \right|^2 \, dx \\
        &\leq \int_E \left( \sum_{\lambda_j \in S} |e_j(x)|^2 \right) \left( \sum_{\lambda_j \in S} | \widehat f(j) |^2 \right) dx \\
        &\leq \left(\int_E \sum_{\lambda_j \in S} |e_j(x)|^2 \, dx \right) \|f\|_{L^2}^2.
    \end{align*}
\end{proof}


\vskip.125in 

\subsection{Proof of Corollary \ref{cor: homogeneous l2 manifold uncertainty}}

\vskip.125in 

\begin{proof}
    From \eqref{eq: constant l2 condition}, we have
    \[
        \frac{1}{\#\{j : \lambda_j \in S\}} \sum_{\lambda_j \in S} \frac 1 {|E|} \int_E |e_j(x)|^2 = \frac{1}{|M|}.
    \]
    The corollary follows immediately from Proposition \ref{prop: l2 manifold uncertainty}.
\end{proof}

\vskip.125in 

\subsection{Proof of Theorem \ref{theorem: randomshitonmanifolds}} We shall need the following result due to Guedon, Mendelson, Pajor, and Tomczak-Jaegermann (\cite{GMPT08}). 

\begin{theorem} \label{theorem: GMPT08} There exist two positive constants $c$ and $C$ such that for any even integer $n$ and any uniformly bounded orthonormal system $(\phi_j)_{j=1}^n$  in $L^2$ with 
\[
\|\phi_j\|_\infty \leq B \quad \text{for } 1 \leq j \leq n,
\]
we can find an index  subset  $I \subset \{1, \ldots, n\}$ with
\[ \left| \# I-   \frac{n}{2}\right|  
 \leq c \sqrt{n}
\]
such that for every $a = (a_i) \in \mathbb{C}^n$,
\begin{equation} \label{equation: GMPT08} {\left\| \sum_{i \in I} a_i \phi_i \right\|}_{L^2} \leq C B \log n (\log \log n)^{5/2} \left\| \sum_{i \in I} a_i \phi_i \right\|_{L^1} \end{equation} 
and
\[
\left\| \sum_{i \notin I} a_i \phi_i \right\|_{L^2} \leq C B \log n (\log \log n)^{5/2} \left\| \sum_{i \notin I} a_i \phi_i \right\|_{L^1}.
\]
\end{theorem} 

Moreover, the proof of this result shows that it holds for a generic set $I$ with the stated property. 

\vskip.125in 

To apply this result, consider the collection of eigenfunctions $\{e_1, \dots, e_n\}$, and apply Theorem \ref{theorem: GMPT08} to this collection. Let 
$$ f(x)=\sum_{i \in I} a_i e_i.$$ 

\vskip.125in 

Suppose that $f(x)$ is $L^1$-concentrated in $E \subset M$ at level $A$. Then, using (\ref{equation: GMPT08}), we have 
\begin{align*}
    {\|f\|}_{L^2(E)} &\leq {\|f\|}_{L^2(M)} \\
    &\leq C B \log n (\log \log n)^{\frac{5}{2}} {\|f\|}_{L^1(M)} \\
    &\leq C  B   A  ~  \log n \cdot (\log \log n)^{\frac{5}{2}} \cdot {||f||}_{L^1(E)} ~.
\end{align*}
It follows that 
$$ {|E|}^{\frac{1}{2}} {\left( \frac{1}{|E|} \int_E {|f(x)|}^2 dx \right)}^{\frac{1}{2}}  \leq  CBA  \cdot \log n \cdot (\log \log n)^{\frac{5}{2}} \cdot |E| \cdot \left( \frac{1}{|E|} \int_E |f(x)| dx \right). $$
The conclusion of Theorem \ref{theorem: randomshitonmanifolds} follows by collecting terms, since 
$$\frac{1}{|E|} \int_E |f(x)| dx \leq {\left( \frac{1}{|E|} \int_E {|f(x)|}^2 dx \right)}^{\frac{1}{2}} $$ by Cauchy-Schwarz.

\section{Open problems}

We conclude the paper with the following open problems: 

\begin{problem} One of the key questions left open in this paper is whether the eigenvalue size fundamentally affects the Fourier uncertainty principle on compact Riemannian manifolds. We have shown (Corollary \ref{cor: homogeneous l2 manifold uncertainty}) that if $(M,g)$ be a compact, boundary-less Riemannian manifold satisfying the homogeneity condition \eqref{eq: constant l2 condition}, then if $E$ is a measurable subset of $M$ and $S$ is a finite subset of the spectrum of $\sqrt{-\Delta_g}$, $f$ is $L^2$-concentrated in $E$ at level $L$ and $\widehat f$ is $L^2$-concentrated on $X_S$ at level $L'$, then 
\begin{equation} \label{eq: holygrail} (1 - \epsilon - \epsilon')^2 \leq \frac{|E|}{|M|} \cdot \# X_S, \end{equation} where $L = (1 - \epsilon^2)^{-\frac 1 2}$ and $L' = (1 - \epsilon'^2)^{-\frac 1 2}$. 

\vskip.125in 

The question we ask is, does (\ref{eq: holygrail}) hold for general compact Riemannian manifolds, with or without a boundary? 
\end{problem} 

\vskip.125in 

\begin{problem} Our second question, highly related to the first, is whether there is a meaningful variant of Theorem \ref{bourgaintheorem} in the setting of compact Riemannian manifolds without a boundary. More precisely, let $X=\{e_1, e_2, \dots, e_n\}$ denote a set of $n$ eigenfunctions of the Laplace-Beltrami operator $\Delta_g$ on a compact manifold $M$ without a boundary. 

\vskip.125in 

A somewhat modest formulation of the question we have in mind is the following. Let $S$ be a random subset of $\{1,2 \dots, n\}$ of size $\approx \sqrt{n}$. Is it true that there exists a $q>2$ and a constant $K_p$ independent of $n$ (and possibly depending on the size of $S$) such that for any function $f \in L^2(M)$ of the form $f=\sum_{i\in S} a_i e_i$, 
 \begin{align}\label{q-orthogonality}
 {\left\| \sum_{i \in S} a_i e_i \right\|}_{L^q(M)} \leq K_p {\left( \sum_{i \in S} {|a_i|}^2 \right)}^{\frac{1}{2}},
 \end{align}
with $K_p$ independent of $n$? Notice that the answer to this question is affirmative when $q=2$. However, our question asks  for the existence  of $p>2$ such that   the ``$q$-orthogonality" relation \eqref{q-orthogonality}
holds. A variant of this result is described in Theorem \ref{theorem: randomshitonmanifolds} above. 
\end{problem}

\end{document}